\documentclass{amsart}

\newtheorem{theorem}{Theorem}[section]

\newtheorem{cor}[theorem]{Corollary}
\newtheorem{prop}[theorem]{Proposition}

\theoremstyle{definition}

\newtheorem{example}[theorem]{Example}

\theoremstyle{remark}
\newtheorem{remark}[theorem]{Remark}

\numberwithin{equation}{section}

\newcommand{\R}{\mathbb{R}}
\newcommand{\Z}{\mathbb{Z}}
\DeclareMathOperator{\supp}{supp}
\def\ds{\displaystyle}
\newcommand{\matrice}[2]{\left(  \begin{array}{#1}#2  \end{array}\right)}
\def\F{\mathcal{F}}
\def\W{\mathcal{W}}
\def\S{\mathcal{S}}
\def\e{\varepsilon}
\DeclareMathOperator{\m}{m}

\begin{document}

\title{On measures driven by Markov chains}

\author{Yanick Heurteaux}
\address{Laboratoire de Math\'ematiques, Clermont Universit\'e\\
Universit\'e Blaise Pascal and CNRS UMR 6620\\ 
BP 80026, 63171 Aubi\`ere, France}
\email{heurteau@math.univ-bpclermont.fr}

\author{Andrzej Stos}
\address{Laboratoire de Math\'ematiques, Clermont Universit\'e\\
Universit\'e Blaise Pascal and CNRS UMR 6620\\ 
BP 80026, 63171 Aubi\`ere, France}
\email{stos@math.univ-bpclermont.fr}

\date{}

\begin{abstract}

We study measures on $[0,1]$ which are driven by a finite Markov chain and which generalize the famous Bernoulli products.
We propose a hands-on approach to determine the structure function $\tau$
and to prove that the multifractal formalism is satisfied. 
Formulas for the dimension of the measures and for the Hausdorff dimension of their supports are also provided.

\end{abstract}

\maketitle

\section{Introduction}

Multifractal measures on $\R^d$ are measures $\m$ for which the level sets
$$E_\alpha=\left\{x\in\R^d\ ;\lim_{r\to 0}\frac{\log\m(B(x,r))}{\log r}=\alpha\right\}$$
are non-trivial for at most two values of the real $\alpha$. In practice, it is impossible to completely describe the sets $E_\alpha$, but we can try to calculate their Hausdorff dimensions. To this end, Frisch and Parisi (\cite{FP85}) were the first to use the Legendre transform of a structure function $\tau$. A mathematically rigorous approach was given by Brown, Michon and Peyri\`ere in \cite{BMP} and by Olsen in \cite{O95}. There are many situations in which the Legendre transform formula, now called multifractal formalism, is satisfied. For a comprehensive account see see e.g. \cite{CM92}, \cite{FL02}, \cite{FO03}, \cite{T06} or \cite{H07}.

A fundamental model for multifractal measures is given by so called Bernoulli products, see for example Chapter 10 of \cite{F2}. Roughly speaking, if $I_{\e_1\cdots\e_n}$ are the $\ell$-adic intervals of the $n^{th}$ generation and $(X_n)$ is an i.i.d. sequence of random variables, then Bernoulli product $\m$ can be defined by 
$$\m(I_{\e_1\cdots\e_n})=P[X_1=\e_1,\cdots,X_n=\e_n].$$ 

The purpose of this paper is to provide an explicit analysis of a natural generalization of this model. Instead of an i.i.d. sequence, we consider an irreducible homogeneous Markov chain. Consequently, the measure $\m$ satisfies the recurrence relation
$$\m(I_{\e_1\cdots\e_{n+1}})=p_{\e_n\e_{n+1}}\m(I_{\e_1\cdots\e_n}),$$
where $P=(p_{ij})$ is the transition matrix of $X_n$ (see the next section for full  details).

In Section \ref{sectau}, we identify a formula for the structure function $\tau$ of such a measure and we compute its dimension. 
Section \ref{multi} contains a construction of auxiliary (Gibbs) measures. While it involves a nontrivial rescaling, we insist on the fact that our results don't require sophisticated tools but only some fundamental results of multifractal analysis and the use of Perron-Frobenius theorem.
We are then able to prove that the multifractal formalism is satisfied and we give a formula for the Hausdorff dimension of the (closed) support of the measures.
Finally, in Section \ref{secergodic} we discuss ergodic properties and prove that for a given support $K$, the measure $\m$ with maximal dimension is essentially unique. 

\section{Preliminaries}\label{prelim}
Set $\S=\{0,1,2, \ell-1\}$.
Let $\W_n$ be set of words of length $n$ over the alphabet $\S$.
A concatenation of two words $\e=\e_1\cdots\e_n\in\W_n$ and 
$\delta=\delta_1\cdots\delta_k\in\W_k$ 
will be denoted by $\e\delta=\e_1\cdots\e_n\delta_1\cdots\delta_k$.
Let $\F_0=\{ [0,1)\}$ and for $n\ge 1$, $\F_n$ be the set of  $\ell$-adic intervals of order $n$, that is
the family of intervals of the form
\[ 
I_\e=I_{\e_1\cdots \e_n}=\left[ \sum_{i=1}^n \frac{\e_i}{\ell^n}, 
\sum_{i=1}^n \frac{\e_i}{\ell^n}+\frac{1}{\ell^n} \right)
\]
where $\e=\e_1\cdots\e_n\in \W_n$. 
 If $I=I_\e\in\F_n$ and $J=I_\delta\in\F_k$, we will write $IJ=I_{\e\delta}$. Note that $IJ\subset I$. Finally, we will denote by $I_n(x)$ the unique interval $I\in\F_n$ such that $x\in I$ and $\vert I\vert$ the length of the interval $I$.

Consider a discrete Markov chain $X=(X_k)_{k\ge 1}$ on $\S$ 
with an initial distribution $p_i=P(X_1=i)$, $i\in \S$, and 
the transition probabilities $P=(p_{ij})_{i,j=0}^{\ell-1}$ where
$p_{ij}=P(X_{n+1} = j \: |\:  X_{n}=i)$. In order to exclude degenerated cases, we will suppose that $p_{ij}\not=1$ for any $(i,j)\in\S\times\S$. 
Note that for the sake of coherence with the definition of $\F_n$, 
the entries of matrices and vectors will be indexed by numbers starting from 0. 

If necessary, we will assume that $X$ is irreducible, that is
for all $i,j\in S$ there exists $k\ge 2$ such that $P(X_k=j | X_1=i) >0$.

For a matrix $P=(p_{ij})_{i,j\in\S}$ we denote by $P^n$ the usual matrix power.
This should not be confused with $P_q$, $q\in\R$, which stands for 
the matrix whose entries are given by $(p_{ij}^q)_{i,j\in\S}$. 
In this context, by convention we will set $0^q=0$ for any $q\in\R$.

The Hausdorff dimension (box dimension, respectively) of a set $E$ 
will be denoted by $\dim_H E $ ($\dim_B E$). 
Recall the definitions of the lower and upper dimension of a probability 
measure $\mu$ on $\R^d$:
\begin{align*}
\dim_*\mu & = \inf\{\dim_H(E) : \mu(E)>0  \} = \sup\{s>0 : \mu\ll\mathcal{H}^s \} \\ 
\dim^* \mu & = \inf\{\dim_H(E) : \mu(E)=1 \} = \inf\{s>0 : \mu\perp\mathcal{H}^s \}
\end{align*}
where $\mathcal{H}^s$ denotes the Hausdorff measure in dimension $s$.
If these two quantities agree, the common value is called the dimension of $\mu$
and denoted by $\dim \mu $. In this case the measure is said to be {\it unidimensional}.
For more details, we refer the reader to e.g. \cite{H07}.

By $c$ or $C$ we will denote a generic positive constant whose exact value is not important 
and may change from line to line.
For functions or expressions $f$ and $g$, depending on a variable $x$, say,
we will write $f\asymp g$ if there exists a constant $c$ which does not depend on $x$ and such that 
$c^{-1}g(x) \le f(x) \le cg(x)$ for any admissible $x$.

\section{The measures and the structure function $\tau$}\label{sectau}

Let $(X_n)_{n\ge 1}$ be a Markov chain on $\S$ with initial distribution $(p_i)_{i\in\
S}$ and transition matrix $P=(p_{ij})$. Define the measure $\m$ as follows:
\[ 
    \m(I_{\e_1\cdots\e_n}) = P(X_1=\e_1,\ldots, X_n=\e_n).
\]
Then $\m(I_i)=p_i$, $i\in \S$, and by the Markov property,
\begin{align*}
    \m(I_{\e_1\cdots\e_n}) &= 
    P(X_n=\e_n\:|\: X_{n-1}=\e_{n-1}\cdots X_1=\e_1)
    P(X_{n-1}=\e_{n-1}\cdots X_1=\e_1)\\
    & = 
    p_{\e_{n-1}\e_n} \m(I_{\e_1\cdots\e_{n-1}}).
\end{align*}
Iterating, we get 
\begin{equation}\label{m_i}
    \m(I_{\e_1\cdots\e_n}) =  p_{\e_1}p_{\e_1\e_2}\cdots p_{\e_{n-1}\e_n}.
\end{equation}
In other words, a finite trajectory of $X$ selects an interval and assigns it
a mass equal to the probability of the trajectory. The measure is well defined 
since for any given $\ell$-adic interval $I$ there is exactly one path to it 
and its subintervals can be reached  only through $I$.

Because of the additivity property $\m(I_{\e_1\cdots\e_n})=\sum_{k=0}^{n-1}\m(I_{\e_1\cdots\e_{n-1}k})$ and the property $\lim_{n\to +\infty}\m(I_{\e_1\cdots\e_n})=0$, it is well known that formula \eqref{m_i} defines a probability Borel measure whose support is contained in $[0,1]$ (see for example \cite{T08}). Moreover the measure $\m$ is such that $\m(\{x\})=0$ for any point $x$.  

The construction proposed here can be viewed as a generalization of a classical Bernoulli measure. Our goal is to give a hands-on approach to the multifractal analysis of such measures.

\begin{example}\label{Bern}
\begin{enumerate}
\item Natural measure on the triadic Cantor set.
        Let 
        \[ P = \frac{1}{2}\matrice{ccc}{1 & 0 & 1 \\ 1 & 0 & 1 \\ 1 & 0 & 1 } \]
        and set the initial distribution $(p_0,p_1,p_2)=(1/2,0,1/2)$. 
        Then $\mathcal C=\supp \m$ is the ternary Cantor set
        and $\m$ is the normalized ($\log_3 2$)-Hausdorff measure on $\mathcal C$.

    \item Bernoulli measures.
        Let $p=(p_0,\cdots ,p_{\ell-1})$ be a probability vector and suppose that $(X_i)$ are i.i.d. random variables 
        with $P(X_1=j) = p_j$. By independence
        $p_{ij}= p_j$ so that
        \[P=\matrice{cccc}{p_0 & p_1 & \cdots & p_{\ell -1} \\
	  \vdots & & & \vdots\\
	  \vdots & & & \vdots\\
	  p_0 & p_1 & \cdots & p_{\ell -1}}. \]
        Hence, by \eqref{m_i} we get 
        $\m(I_{\e_1\cdots\e_n}) = p_{\e_1}\cdots p_{\e_n}$,
	and the measure $\m$ is the classical Bernoulli product.

    \item Let $\ell=2$, $p\not=1/2$ and
	\[P=\matrice{cc}{p &1-p  \\ 1-p & p } .\]
	The associated measure $\m$ was introduced by Tukia in \cite{T89}. It is a doubling measure with dimension $\dim(\m)=-(p\log_2 p+(1-p)\log_2(1-p))<1$. The associated repartition function $f(x)=\m([0, x])$ is a singular quasisymetric function.
    \item Random walk on $\Z_\ell$.  Let $p_{ij}=1/2$ if $|i-j|= 1$ 
        or $(i,j)=(0,\ell-1)$ or $(i,j) = (\ell-1,0)$.
        If, for example, $n=3$, 
        \[ P = \frac{1}{2}\matrice{ccc}{0 & 1 & 1 \\ 1 & 0 & 1 \\ 1 & 1 & 0 }. \]
        As we will see later, the associated measure is monofractal with dimension $\frac{\log 2}{\log \ell}$.

\end{enumerate}
\end{example}

\vskip 0.5cm
Define as usual the structure function $\tau(q)$ by 
\begin{equation}\label{tau}
    \tau(q) = \limsup_{n\to\infty} \frac{1}{n\log \ell}
    \log\left( \sum_{I\in\F_n}\m(I)^q \right),
\end{equation}
with the usual convention $0^q=0$ for any $q\in\R$.
\begin{theorem}\label{tau-lambda}
    Let $\m$ be a probability measure driven by a Markov chain with transition matrix $P$. Suppose that the matrix $P$ is irreductible. For $q\in\R$, let $\lambda_q$ be the spectral radius of $P_q$. 
    Then $\tau(q) = \log_{\ell}(\lambda_q)$ and the limit does exist in \eqref{tau}.
\end{theorem}
\begin{proof}
    Let $\W_{n,k}$ be the subset of $\W_n$ consisting of words that end with 
    $k\in\S$. Set 
    \[ 
        s_{n,k} = \sum_{\e\in \W_{n,k}} \m(I_\e)^{q}  
    \]
    and let $S_n$ be the (line) vector $(s_{n,0},\ldots,s_{n,\ell-1})$.
    In particular, $\W_{1,k} = \{k\}$ and $S_1=(p_0^q,\ldots,p_{\ell-1}^q)$.
    We claim that
    \begin{equation}\label{S_n}
        S_nP_q = S_{n+1}, \qquad n\ge 1.
    \end{equation}
    Indeed, the $j^\text{th}$ coordinate of $S_nP_q$ is given by
    \begin{equation*}\label{jth}
        \sum_{k\in\S}s_{n,k}  p_{kj}^q 
         = \sum_{k\in\S} \sum_{\e\in \W_{n,k}} p_{kj}^q \m(I_\e)^{q}  
    \end{equation*}
    Using the Markov property, we have $p_{kj} \m(I_\e)=\m(I_{\e j})$ when $\e\in \W_{n,k}$. It follows that 
    \[
        \sum_{k\in\S}s_{n,k}  p_{kj}^q =\sum_{k\in\S} \sum_{\e\in \W_{n,k}} \m(I_{\e j})^q 
        = \sum_{\e\in \W_{n+1,j}} \m(I_{\e})^q=s_{n+1,j}
    \]
    which is the $j^\text{th}$ coordinate of $S_{n+1}$. So that \eqref{S_n} follows.
    Iterating, we obtain 
    \begin{equation}\label{S_1}
        S_n = S_1\left(P_q\right)^{n-1}.
    \end{equation}
    Now, observe that 
    \[
        \sum_{\e\in \W_{n}} \m(I_\e)^q
        = 
        \sum_{k\in\S} s_{n,k} 
        =  
        \Vert S_n\Vert_1 
        =  
        \Vert S_1P_q^{n-1}\Vert_1  
    \]
    where $\Vert\ \Vert_1$ is the $\ell^1$ norm in $\R^\ell$.
    It follows that 
    \begin{align*}
        \tau(q) 
        & = 
        \limsup_{n\to\infty} \frac{1}{n}\log_{\ell}
        \left(\sum_{\e\in W_n} \m(I_\e)^q\right) \\
        & = 
        \limsup_{n\to\infty} \frac{1}{n}\log_{\ell} \| S_1P_q^{n-1}\|_1\\
        & = 
        \limsup_{n\to\infty} \frac{1}{n}\log_{\ell} \| S_1P_q^{n}\|_1 
    \end{align*}
    Let us now introduce the following notation. If $a=(a_0,\cdots,a_{\ell-1})$ and $b=(b_0,\cdots,b_{\ell-1})$ are two vectors in $\R^\ell$, we will write $a\prec b$ when $a_i\le b_i$ for any value of $i$.
    Observe in particular that that $S_1\succ 0$ and is not identically equal to 0. Let $\lambda_q$ be the spectral radius of the matrix $P_q$, which is also the spectral radius of the transposed matrix $P_q^t$. The matrix $P_q$ being positive and irreductible, Perron-Frobenius Theorem ensures the existence of an eigenvector $\nu_q$ with strictly positive entries, satisfying $\nu_qP_q=\lambda_q\nu_q$. Therefore,  there exists a constant $C>0$ such that $S_1\prec C\nu_q$. It follows that
    $$\Vert S_1P_q^{n}\Vert_1\le C\Vert \nu_qP_q^n\Vert_1=C\Vert\nu_q\Vert_1\lambda_q^n.$$
    On the other hand, using the irreductibility property of the matrix $P_q$, we can find an integer $k$ such that the matrix $I+P_q+\cdots +P_q^k$ has strictly positive entries. It follows that the line vector $S_1+S_1P_q+\cdots +S_1P_q^k$ has strictly positive entries and we can find a constant $C>0$ such that 
    $$\nu_q\prec C\left(S_1+S_1P_q+\cdots +S_1P_q^k\right).$$
    So
    \begin{align*}
	\Vert \nu_qP_q^n\Vert_1
	& \le
	C\Vert S_1P_q^n+S_1P_q^{n+1}+\cdots + S_1P_q^{n+k}\Vert_1\\
	& =
	C\Vert \left(S_1P_q^n\right)\left( I+P_q+\cdots P_q^k\right)\Vert_1\\
	& \le
	C'\Vert S_1P_q^n\Vert_1.
    \end{align*}
    Finally, $\Vert S_1P_q^{n}\Vert_1\asymp \lambda_q^n$.Taking the logarithm, we can conclude that $\tau(q)=\log_{\ell}(\lambda_q)$ and that the limit exists.

\end{proof}

\begin{cor}\label{analytic}
The function $\tau$ is analytic on $\R$.
\end{cor}
  \begin{proof}
This can be seen as a consequence of the Kato-Rellich theorem (see for example \cite{RS}). But in this finite dimensional context, there is an elementary proof. Let $F(q,x)=\det(P_q-xI)$ be the characteristic polynomial of $P_q$ and let $q_0\in\R$. Observing that $F(q_0,\lambda_{q_0})=0$ and $\frac{\partial F}{\partial x}(q_0,\lambda_{q_0})\not= 0$ (the eigenvalue $\lambda_{q_0}$ is simple), the map $q\mapsto\lambda_q$ is given arround $q_0$ by the implicit functions theorem. Moreover, $F$ being analytic in $q$ and $x$, it is well known that the implicit function is analytic.
  \end{proof}

The existence of $\tau'(1)$ ensures that the measure $\m$ is unidimensional (see e.g. \cite{H07}, Theorem 3.1).
\begin{cor}\label{tauprime}
 The measure $\m$ is unidimensional with dimension $\dim(\m )=-\tau'(1)$.
\end{cor}

Let us now describe some examples.
Let $h_\ell$ be the usual  entropy function 
\[ 
  h_\ell(p) = -\sum_{i=0}^{\ell-1} p_i\log_\ell p_i,\qquad p=(p_0,\ldots, p_{\ell-1}) 
  \text{ with }\sum_{i=0}^{\ell-1} p_{i}=1.
\]
In particular, set $h(x) = h_2(x,1-x) = -(x\log_2(x) + (1-x)\log_2(1-x))$.
\begin{example}
    If $\m$ is the Bernoulli measure from Example \ref{Bern} (2),
    then by Theorem \ref{tau-lambda} we get the well known formula for $\tau$
    \[\tau(q) = \log_\ell\left(p_0^q+\cdots +p_{\ell-1}^q\right).\] 
    Furthermore, the dimension of the measure $\m$ is $\dim(\m) = -\tau'(1)=h_\ell(p)$.
\end{example}

\begin{example}
    Actually, if $\ell=2$, we can obtain an explicit formula 
    for any given Markov chain.
    Suppose that $a,b\in(0,1)$ and \[P= \matrice{cc}{1-a & a \\ b & 1-b}.\]
    Then 
\begin{equation}\label{dim2}
    \dim(\m) = \frac{b}{a+b} h(a) + \frac{a}{a+b}h(b).
\end{equation}
    Indeed, by Theorem \ref{tau-lambda} we get 
    \[ 
    \tau(q) = -1+ \log_2 \left(
        (1-a)^q+(1-b)^q+\sqrt{ ((1-a)^q-(1-b)^q)^2+4a^qb^q } \right).
    \]
    Note that if $q=1$, then $\sqrt{ ((1-a)^q-(1-b)^q)^2+4a^qb^q }$ simplifies to 
    $a+b$. Thus, we obtain
    \begin{multline*}
    \tau'(1) = \frac{1}{2\log 2}  \Bigg( (1-a)\log (1-a)+(1-b)\log (1-b) \;+ \\ 
        +\frac{ (b-a)((1-a)\log (1-a) - (1-b)\log (1-b)) + 2ab\log ab}{a+b}
    \Bigg) 
    \end{multline*}
    Rearranging, we get \eqref{dim2}. Note that the coefficients in \eqref{dim2} come from 
    the stationary distribution of $X$:
    \[ \pi=\left( \frac{b}{a+b},  \frac{a}{a+b}\right).\]
    This observation will be generalized in Theorem \ref{dot-prod} below.
\end{example}

\begin{example}\label{uniform-cantor}
	Let $(a_0,\cdots,a_{\ell-1})$ be a probability vector (with possibly some entries that are equal to 0) and  $P$ be an $\ell\times \ell$ irreductible stochastic matrix with entries $a_0,\ldots, a_{\ell-1}$ in every row, but in an arbitrary order.
	Then $\tau(q) =  \log_\ell\left(a_0^q+\ldots+a_{\ell-1}^q\right)$ and $\dim(\m)=h_\ell\left(a_0,\cdots,a_{\ell-1}\right)$.
	In particular, if $\kappa$ is the number of $a_i$'s that are not equal to 0 and if each nonzero $a_i$ is equal to $1/\kappa$, we get $\dim(\m)=\frac{\log\kappa}{\log\ell}$ which is the maximal possible value and is also the dimension of the support of the measure $\m$. Such a remark will be generalized below.
\begin{proof}
    Set 	$\ds A_q = a_0^q+\cdots+ a_{\ell-1}^q $.
	We have
   \begin{align*}
      \sum_{\e\in \W_{n+1}} \m(I_\e)^q
      & =  \sum_{k\in\S} \sum_{\e\in\W_{n,k}} \sum_{j\in\S} p_{kj}^q\m(I_{\e })^q \\
      & =  A_q  \sum_{k\in\S} \sum_{\e\in\W_{n,k}} \m(I_{\e })^q \\ 
      & =  A_q\sum_{\e\in\W_n} \m(I_{\e })^q.  \\
    \end{align*}
    Iterating, we get
    \[  \sum_{\e\in \W_n} \m(I_\e)^q = A_q^{n-1} \sum_{\e\in \W_1} \m(I_\e)^q .  \]
    Consequently, 
    \[ \tau(q) =  \log_\ell\left(a_0^q+\ldots+a_{\ell-1}^q\right). \]
    It follows that  $ \tau'(1) = -h_\ell\left(a_0,\ldots,a_{\ell-1}\right) $.

If there are only $\kappa$ nonzero values in the probability vector $(a_0,\cdots,a_{\ell-1})$, say e.g. $a_0,\cdots,a_{\kappa-1}$, the formula turns to $\dim\m=-\sum_{j=0}^{\kappa-1}a_j\log_\ell a_j$ which is maximal when $a_j=1/\kappa$ for any $j$.
    \end{proof}  
\end{example}

Let us finish this part with a general formula for the dimension of the measure $\m$. That is the purpose of the following theorem.

\begin{theorem}\label{dot-prod}
	Denote by $L_k$ the $k^{\text{th}}$ line of the matrix $P$. 
	Let $H=\left(h_\ell(L_0),\ldots, h_\ell(L_{\ell-1})\right)$ be the vector of entropies of the lines of $P$. Then 
	\[ \dim(\m ) = <\pi\,\vert\, H>,\] 
	where $\pi$ is the stationary distribution of the Markov chain $X$ and $<\ \vert\ >$ is the canonical scalar product.
\end{theorem}
\begin{proof}
Let $\ds T_n = \sum_{I\in\F_n} \m(I)\log_\ell \m(I)$, and $H_n=\frac{-1}nT_n$ the entropy related to the partition $\F_n$.
We know that $\tau'(1)$ exists. It follows that $\ds  -\tau'(1) = \lim_{n\to +\infty} H_n$ (see for example \cite{H07}, Theorem 3.1).
Further, 
\begin{align*}
  T_n & = \sum_{k\in\S}\sum_{\e\in\W_{n-1,k}}\sum_{j\in\S} \m(I_\e)p_{kj}\log_\ell(\m(I_\e)p_{kj}) 
  \\ & =
  -\sum_{k\in\S} h_\ell(L_k) \sum_{\e\in\W_{n-1,k}} \m(I_\e) +
  \sum_{k\in\S}\sum_{\e\in\W_{n-1,k}} \m(I_\e)\log_\ell(\m(I_\e)) 
\end{align*}
Denote, as before, $\displaystyle s_{n-1,k}=\sum_{\e\in\W_{n-1,k}} \m(I_\e) $.
It follows that
\[
  T_n = -\sum_{k\in\S} h_\ell(L_k) s_{n-1,k} + T_{n-1}.
\]
Iterating, we get
\[
  \tau'(1) = \lim_{n\to +\infty}  \sum_{k\in\S} h_\ell(L_k) \frac1n (s_{n-1,k} + s_{n-2,k} + \ldots + s_{1,k})
\]
Observe now that $s_{n-1,k}$ is the $k^{\text{th}}$ component of $S_{n-1}=S_1P^{n-2}$ (see \eqref{S_1}). 
It is well known that the Cesaro means converge to the stationary distribution $\pi$ (even in the periodic case). The theorem follows.
\end{proof}

\begin{example}\label{non-uniform}
  Let 
  \[ P = \matrice{cccc}{1/3 &0 &1/3 &1/3\\ 0 &1/2 &0 &1/2 \\ 
           1/2 &1/2&0&0 \\ 1/2 &1/2&0&0}.  
   \] 
   An analytic formula for $\tau(q)$ is complicated. 
   Nevertheless, a numerical evaluation of $\tau(0)$ (which is also the box dimension of the support of $\m$) is possible and Theorem \ref{dot-prod} allows us to estimate the dimension of the measure $\m$. We find
   $$\dim \m \approx 0.58\qquad\mbox{and}\qquad\dim_B(\supp \m) \approx 0.60.$$ 
   This example shows that the ``uniform'' transition densities does not need to imply the maximality of $\dim\m$. Indeed, we will see in Corollary \ref{maximal} that there always exists a choice of the transition matrix for which the dimension of the measure coincides with the dimension of its support.
\end{example}

\section{Multifractal analysis and dimension of the support}\label{multi}

By the definition of $\tau$ we have 
\[\tau(0) = \limsup_{n\to\infty} \frac{\log N_{n}}{n\log\ell},\]
where $N_{n}$ is the number of intervals from $\F_n$ having 
positive measure. In our context the limit exists, so $\tau(0) = \dim_B(\supp\m)$.
Observe that the support of the measure $\m$ doesn't depend on the specific values 
of $p_{ij}$ but only on the configuration of the nonzero entries
in the matrix $P_0$ and in the initial distribution $p=\left( p_0,\cdots,p_{\ell-1}\right)$ . More precisely, the support of the measure $\m$ is the compact set
$$K=\bigcap_{n\ge 1}\bigcup_{p_{\e_1}p_{\e_1\e_2}\cdots p_{\e_{n-1}\e_n}>0}\overline{I_{\e_1\cdots\e_n}}.$$
Indeed, the construction of the support can be viewed as a Cantor-like removal process.
Given an  interval $I_{\e_1\cdots\e_n}$ of the $n^\text{th}$ generation with $\e_n=i$, 
its $j^\text{th}$ subinterval will be removed if and only if $p_{ij}=0$
(cf. Example \ref{Bern} (1)).

According to Theorem \ref{tau-lambda}, $\overline{\dim}_B(\supp\m)=\log_\ell\lambda_0$ where $\lambda_0$ is the spectral radius of the matrix $P_0$, so that the box dimension of $\supp\m$ does not depend on the initial distribution $p$ and only depends on the configuration of the nonzero entries of the matrix $P_0$.

This motivates the following questions.
Given a configuration of nonzero entries of $P_0$ and of the initial distribution $p$, 
which values of $p_{ij}$  maximize $\dim \m$? Is the maximal value of $\dim\m$ equal to the box dimension of the support ?
Is this maximal measure unique?

In some cases one has an immediate answer. In particular, Example \ref{uniform-cantor} says that if each row of the matrix $P$ has the same number $\kappa$ of nonzero entries, the maximum of $\dim\m$ is obtained when each nonzero entry of the matrix $P$ is equal to $1/\kappa$ and is then equal to $\frac{\log\kappa}{\log\ell}$. 

The general answer will be a consequence of the following result which says that the measure $\m$ satisfies the multifractal formalism.

\begin{theorem}\label{multif}
    Let $X$ be an ireductible Markov chain with transition matrix $P$ and let $\m$ be the 
    associated measure. Then $\m$ satisfies the multifractal formalism. More precisely, define
$$E_\alpha=\left\{ x\in[0,1]\ ;\ \lim_{n\to \infty}\frac{\log\m(I_n(x))}{\log\vert I_n(x)\vert}=\alpha\right\}\ .$$
Then, for any $-\tau'(+\infty)<\alpha<-\tau'(\infty)$, 
$$\dim(E_\alpha)=\tau^*(\alpha),$$
where $\tau^*(\alpha)=\inf_q(\alpha q+\tau(q))$ is the Legendre transform of the function $\tau$.
\end{theorem}
\begin{proof}
    We will prove the existence of a Gibbs measure at a given state $q$, 
    that is an auxiliary measure $\m_q$
    such that for any $\ell$-adic interval $I$ one has  
    \begin{equation}\label{gibbs}
    \m_q(I) \asymp |I|^{\tau(q)} \m(I)^q.
    \end{equation}
    Note that in \eqref{gibbs}, the constant may depend on $q$. Since the function $\tau$ is differentiable, it is well known that the existence of such a measure at each state $q$  
    implies the validity of the multifractal formalism for $\m$ (see for example \cite{BMP} or \cite{H07}).
    
    Now again, such a Gibbs measure will be obtained with an elementary construction. Note that $P$ is irreductible  
    if and only if $P_q$ is.
    By Perron-Frobenius Theorem, the spectral radius $\lambda_q$ of $P_q$ is a simple eigenvalue and there exists a unique probability vector  $\pi_q$ with strictly positive entries and satisfying $P_q\pi_q=\lambda_q\pi_q$. 
    
    Define $D_q$ as the $\ell\times\ell$-matrix having the coordinates of $\pi_q$ 
    on the diagonal and zeros elsewhere.
    Set $Q_q = \dfrac{1}{\lambda_q}D_q^{-1}P_q D_q$ and let $\mathbf{1}$ be the column vector
    $\mathbf{1}= (1,\ldots,1)^t\in\R^\ell$.
    Then we have 
    \[ Q_q\mathbf{1}  =  \frac{1}{\lambda_q}D_q^{-1}P_qD_q {\mathbf{1}}
        = \frac{1}{\lambda_q}D_q^{-1}P_q \pi_q= D_q^{-1}\pi_q = {\mathbf 1} .\]
    In other words, $Q_q$ is a stochastic matrix and thus it can be associated to a Markov chain 
    $X^{(q)}$.
    We may and do assume that this chain has the initial distribution 
    $q_i=\alpha p_i^q$, $i\in\S$, where $(p_0,\ldots,p_{\ell-1})$ is the initial 
    distribution of $X$ and $\alpha=(p_0^q+\cdots +p_{\ell-1}^q)^{-1}$ is the normalization constant.  Let $\m_q$ be the measure induced by $X^{(q)}$.
    Remark that if $Q_q=(q_{ij})$, $D_q=(d_{ij})$, then we have
    \[q_{ij} = \dfrac{1}{\lambda_q}d_{ii}^{-1} p_{ij}^q d_{jj}, \qquad i,j\in\S. \] 
    Cleraly, $q_{ij}>0 $ if and only if $p_{ij}>0$ and $q_i>0$ if and only if $p_i>0$.
    Therefore the measures $\m_q$ and $\m$ have the same support.
        
    Let $I=I_{\e_1\cdots\e_n}\in\F_n$. We have
    \begin{align*}
        m_q(I) 
        &=  
        q_{\e_1}q_{\e_1\e_2} \cdots q_{\e_{n-1}\e_{n}}
        \\ &=  
        \alpha p_{\e_1}^q\left(\frac{1}{\lambda_q}d_{\e_1\e_1}^{-1}p_{\e_1\e_2}^q d_{\e_2\e_2}\right)
        \cdots
        \left(\frac{1}{\lambda_q}d_{\e_{n-1}\e_{n-1}}^{-1}p_{\e_{n-1}\e_n}^qd_{\e_n\e_n}\right)
        \\ &=
        \frac{\alpha}{\lambda_q^{n-1}}d_{\e_1\e_1}^{-1} m(I_{\e_1\cdots\e_n})^qd_{\e_n\e_n}
    \end{align*}
    Further, since the entries of the the eigenvector $\pi_q$ are strictly positive, 
    there is a constant $c$ (possibly depending on $q$) such that 
    \[ c^{-1}\le\frac{d_{ii}}{d_{jj}}\le c\] 
    for any $i,j\in\S$. 
    This yields \[ m_q(I) \asymp  \frac1{\lambda_q^{n-1}} m(I)^q.\]
    Now, observe that by Theorem \ref{tau-lambda}, $\ell^{\tau(q)}=\lambda_q$ so that
    \[ |I|^{\tau(q)}=\left(\ell^{-n}\right)^{\tau(q)} = \lambda_q^{-n}.\] 
    It follows that for any $I\in\F_n$ we have 
    \[ m_q(I) \asymp |I|^{\tau(q)} m(I)^q.\]
    Note that in the above estimate the implicit constants may depend on $q$ 
    but not on $n$ and $I$. Therefore $m_q$ is the needed Gibbs measure and the theorem follows.
\end{proof}
\begin{cor}\label{maximal}
    Let $K$ be the support of the measure $\m$. The measure $\overline\m=\m_0$ associated to the matrix $Q_0$ satisfies $\supp\overline\m=K$, is monofractal and strongly equivalent to the Hausdorff measure $\mathcal H^{\tau(0)}$ on $K$. In particular,
    $$\dim(\overline\m)=\dim_H(K)=\dim_B(K)$$
    and $\overline\m$ is a measure driven by a Markov chain, with support $K$ and with maximal dimension.
\end{cor}
\begin{proof}
    Let $\overline\m = \m_0$ from the previous proof.
    Then we have 
    \[ \overline\m(I) \asymp |I|^{\tau(0)}\]
    for any $\ell$-adic interval $I$ such that $\m(I)>0$. In particular, for any $x\in K$, $\overline\m\left(I_n(x)\right)\asymp \left\vert I_n(x)\right\vert^{\tau(0)}$. 
    By Billingsley's theorem (see e.g. \cite{F2}, Propositions 2.2 and 2.3), 
    we conclude that $\overline\m$ is equivalent to $\tau(0)$-dimensional Hausdorff
    measure $\mathcal{H}^{\tau(0)}$ on $K$. 
    In particular $\mathcal{H}^{\tau(0)}(K)$ is positive and finite. It follows that $$\dim_H(K)=\tau(0)=\dim_B(K).$$ 
    On the other hand, the structure function $\overline\tau$ of the measure $\overline\m$ is $\overline\tau(q)=\tau(0)(1-q)$. It follows that 
    the measure $\overline\m$ is monofractal, and that 
    $$\dim\overline\m=-\overline\tau'(1)=\tau(0)=\dim_B(K).$$
\end{proof}
\begin{example}
Suppose that $X$ is a random walk on $\Z_\ell$ 
(cf. Example \ref{Bern} (3)).
Then $\tau(q) = (1-q)\log_\ell 2$ and 
\[\dim \m = -\tau'(1) =  \log_\ell 2 = \tau(0) = \dim_B(K) .\]
\end{example}

\begin{example}\label{non-max} 
        Let 
        \[ P = \matrice{ccc}{1/3 &1/3 &1/3\\ 1/2 &0 &1/2 \\ 1/2 &0 &1/2}.  \] 
        It can be easily seen that $P$ is irreducible (actually, $P^2$ has only positive entries).  We obtain
        \[ \tau(q) = -\log_3 2 + 
            \log_3\left( 2^{-q}+3^{-q} +\sqrt{4^{-q}+6^{1-q}+9^{-q}}\right) . \]
        Hence 
	\[\dim_H(\supp\m)=\log_3(1+\sqrt{2})\approx 0.802\quad\mbox{and}
        \quad\dim \m = \dfrac{1}{7}(3+4\log_3 2)\approx 0.789.\]
As observed in Example \ref{non-uniform}, the transitions are uniform row by row, but the measure $\m$ is not monofractal. The Markov chain inducing the maximal measure $\overline\m$ is associated to the following transition matrix :
\[ 
   Q_0=(\sqrt{2}-1)\times 
  \matrice{ccc}{1 & \sqrt{2}/2 & \sqrt{2}/2 \\ 
  \sqrt{2} &0 &1 \\
  \sqrt{2} &0 &1 }.   
\]
\end{example}

\section{Invariance, ergodicity and application to the uniqueness of $\overline\m$}\label{secergodic}
The goal of this section is to discuss the uniqueness of measure $\overline\m$ with maximal dimension given in Corollary \ref{maximal}. This is the object of Theorem \ref{unicity}.

We need to start with some preliminary results. Let us introduce the shift $\sigma$ on $[0,1)$ defined by $\sigma(x)=\ell x-E(\ell x)$, where $E(y)$ is the integer part of $y$. Observe that $\sigma\left(I_{\e_1\cdots\e_n}\right)=I_{\e_2\cdots\e_n}$ and that 
$$\sigma(x)=\bigcap_nI_{\e_2\cdots\e_n}\quad\mbox{if}\quad x=\bigcap_nI_{\e_1\cdots\e_n}.$$
That is why $\sigma$ is called the shift.

\begin{prop}
 Let $P$ be an irreducible $\ell\times\ell$ transition matrix, $\nu$ be the (unique) probability vector such that $\nu P=\nu$ and $X$ be the Markov chain with transition $P$ and initial law $\nu$. Set $\m_P$ to be the probability measure driven by the Markov chain $X$. Then, $\m_P$ is $\sigma$-invariant and ergodic.
\end{prop}

\begin{proof}
\begin{align*}
 \m_P(\sigma^{-1}(I_{\e_1\cdots \e_n}))
  &=
  \sum_{j=0}^\ell \m_P(I_{j\e_1\cdots\e_n})\\
  &=
  \sum_{j=0}^\ell \nu_jp_{j\e_1}\cdots p_{\e_{n-1}\e_n}\\
  &=
  \nu_{\e_1}p_{\e_1\e_2}\cdots p_{\e_{n-1}\e_n}\\
  &=
  \m_P(I_{\e_1\cdots \e_n})\\
\end{align*}
So, by the monotone class theorem, the measure $\m_P$ is $\sigma$-invariant.

Let $k$ be an integer such that $P+P^2+\cdots+P^k$ has strictly positive entries. 
We claim that there exists a constant $C>0$ such that for any $I,J\in\bigcup_n\F_n$, we have
\begin{equation}\label{ergodic}
\frac1C \m_P(I)\m_P(J)\le\sum_{j=0}^{k-1}\sum_{K\in\F_j}\m_P(IKJ)\le C \m_P(I)\m_P(J).
\end{equation}

Indeed, if $I=I_{\e_1\cdots\e_n}$, $J=I_{\delta_1\cdots\delta_m}$, and if $(\pi_{ij})_{i,j}$ denote the coefficients of the matrix $P+P^2+\cdots+P^k$, it is easy to check that
$$\sum_{j=0}^{k-1}\sum_{K\in\F_j}\m_P(IKJ)=\nu_{\e_1}p_{\e_1\e_2}\cdots p_{\e_{n-1}\e_n}\times\pi_{\e_n\delta_1}p_{\delta_1\delta_2}\cdots p_{\delta_{m-1}\delta_m}$$
and the claim follows.

Inequality \eqref{ergodic} can be rewritten as
\begin{equation}\label{rewritten}
\sum_{j=0}^{k-1}\m_P\left(I\cap\sigma^{-(n+j)}(J)\right)\asymp\m_P(I)\times\m_P(J)
\end{equation}
where $n$ is the generation of $I$. If we observe that any open set is a countable union of disjoint intervals in $\cup_n\F_n$, inequality \eqref{rewritten} remains true when $J$ is an open set. Finaly, by regularity of the measure $\m_p$, it is also true for any Borel set $J$. In particular, if $E$ is a $\sigma$-invariant Borel set, we get 
$$\forall I\in\bigcup_n\F_n,\qquad k\m_P(I\cap E)\asymp\m_P(I)\times\m_P(E).$$
Again, it remains true when $I$ is an arbitrary Borel set. In particular, 
$$\m_P\left( ([0,1]\setminus E)\cap E\right)\asymp \m_P\left( [0,1]\setminus E\right)\times\m_P(E)$$
which proves that $\m_P(E)=0$ or $\m_P\left( [0,1]\setminus E\right)=0$.
\end{proof}

\begin{remark} Inequality \ref{rewritten} is a particular case of the so called weak quasi-Bernoulli property which was introduced by B. Testud in \cite{T06}. 
\end{remark}

\begin{cor}\label{unicity-ergodic}
Let $P$ and $\tilde P$  be two different irreducible $\ell\times\ell$ transition matrices. Then $\m_P$ is singular with respect to $\m_{\tilde P}$.
\end{cor}

\begin{proof}
 According to the ergodic theorem, it suffices to show that $\m_P\not=\m_{\tilde P}$. Let $\nu$ and $\tilde\nu$ be the invariant distributions of the stochastic matrix $P$ and $\tilde P$. If $\nu_i\not=\tilde\nu_i$ for some $i$, then $\m_P(I_i)\not=\m_{\tilde P}(I_i)$. If $\nu=\tilde\nu$ and if $p_{ij}\not=\tilde p_{ij}$, we can write :
$$\m_P(I_{ij})=\nu_ip_{ij}\not=\tilde\nu_i\tilde p_{ij}=\m_{\tilde P}(I_{ij}).$$
\end{proof}

\begin{cor}\label{dichotomy}
Let $P$, $\tilde P$  be two  irreducible $\ell\times\ell$ transition matrices, $p$ and $\tilde p$ to be two probability vectors. Set $\m$ and $\tilde\m$ be the associated measures. Suppose that $\supp(\m)=\supp(\tilde\m)$. Then, there are only two possible cases :
\begin{enumerate}
 \item $P=\tilde P$ and the measures $\m$ and $\tilde\m$ are strongly equivalent (i.e. $\m\asymp\tilde\m$)
  \item $P\not=\tilde P$ and the measures $\m$ and $\tilde\m$ are mutually singular.
\end{enumerate}
\end{cor}

\begin{proof}
 Let $\mathcal A\subset\S$ be the set of ranges of the nonzero entries of $p$ (which is also the set of ranges of nonzero entries of $\tilde p$). Let $F=\bigcup_{\e\in\mathcal A}\overline{I_\e}$. We claim that $\m$ is strongly equivalent to the measure $\m_P$ restricted to $F$. This is an easy consequence of the fact that the invariant probability vector $\nu$ (satisfying $\nu P=\nu$) has strictly positive entries. In the same way, $\tilde\m$ is strongly equivalent to the measure $\m_{\tilde P}$ restricted to $F$. Corollary \ref{dichotomy} is then a consequence of Corollary \ref{unicity-ergodic}. 
\end{proof}

Now, we are able to prove the following theorem on the measure $\overline\m$ given by Corollary \ref{maximal}.

\begin{theorem}\label{unicity}
 Let $P_0=\left(p^0_{ij}\right)$ be an irreductible $\ell\times\ell$ matrix such that $ p^0_{ij}\in\{0,1\}$ for any $ij$ and let $p^0=\left(p^0_0,\cdots,p^0_{\ell-1}\right)$ be a line vector such that $p^0_i\in\{ 0,1 \}$ for any $i$. Suppose that $p^0\not=(0,\cdots,0)$ and define the compact set $K$ by
$$K=\bigcap_{n\ge 1}\bigcup_{p^0_{\e_1}p^0_{\e_1\e_2}\cdots p^0_{\e_{n-1}\e_n}=1}\overline{I_{\e_1\cdots\e_n}}.$$
Let $\delta=\dim_H(K)$ and let $\m$ be a measure with support $K$, driven by a Markov chain $X$ with irreductible transition matrix $P$. Then, $\dim\m=\delta$ if and only if 
$$P=\frac1{\lambda_0}D_0^{-1}P_0D_0$$
where $\lambda_0$ is the spectral radius of $P_0$ and $D_0$ is the diagonal matrix whose diagonal entries are the coordinates of the (unique) probability vector $\pi_0$ satisfying $P_0\pi_0=\lambda_0\pi_0$. Moreover, the case $P=\frac1{\lambda_0}D_0^{-1}P_0D_0$ is the only case where the measure $\m$ is  monofractal.
\end{theorem}

\begin{proof}
 Assume for simplicity that $p_i^0=1$ for any $i$. The general case is a standard modification. Remember that the support of the measure $\m$ only depends on the positions of the non-zero entries of the matrix $P$. It follows the non-zero entries of the matrices  $P$ and $P_0$ hare located  at the same places.  Let $Q_0=\frac1{\lambda_0}D_0^{-1}P_0D_0$. 

Suppose that $P=Q_0$. According to Corollary \ref{maximal} and Corollary \ref{dichotomy}, the measure $\m$ is strongly equivalent to $\overline\m=\m_0$ which  satisfies $\overline\m(I_n(x))\asymp \vert I_n(x)\vert^\delta$ for any $x\in K$. In particular, $\m$  is such that $\dim \m= \delta$ and is monofractal. 

Suppose now that $P\not= Q_0$. Corollary \ref{dichotomy} says that $\m$ is singular with respect to $\m_0$ and we have to prove that $\dim\m<\delta$. Let
$$\tau(q)=\lim_{n\to+\infty}\frac1{n\log\ell}\log\left(\sum_{I\in\F_n}\m(I)^q\right).$$
Recall that $\tau$ is analytic and such that $\tau(0)=\delta$ and $\tau(1)=0$. Using the convexity of $\tau$, it is clear that
$$\tau'(1)=-\delta\Longleftrightarrow \forall q\in[0,1],\ \tau(q)=\delta(1-q)\Longleftrightarrow \forall q\in\R,\ \tau(q)=\delta(1-q).$$
In order to prove that $\dim\m<\delta$, it is then sufficient to establish that $\tau(2)>-\delta$.

Denote by $I_0,\cdots,I_{\ell-1}$ the intervals of the first generation $\F_1$. If $j\in\{ 0,\cdots,\ell-1\}$ and $n\ge 1$,  let $\F_{n}(j)$ be the intervals of $\F_n$ that are included in $I_j$.
The measures $\m$ and $\m_0$ being mutually singular, we know that for $\mbox{d}\!\m$-almost every $x\in K$,
$$\lim_{n\to +\infty}\frac{\m_0(I_n(x))}{\m(I_n(x))}=0,$$
which can be rewritten as
$$\lim_{n\to +\infty}\frac{\ell^{-n\delta}}{\m(I_n(x))}=0.$$
Using Egoroff's theorem in each $I_j$, we can find a set $A\subset K$ such that $\m(A\cap I_j)\ge\frac12 \m(I_j)$ for any $j\in\{0,\cdots,\ell-1\}$ and satisfying
$$\forall\e>0,\quad \exists n_0\ge 1;\quad \forall n\ge n_0,\ \forall x\in A,\quad\m(I_n(x))\ge\frac1\e\ell^{-n\delta}.$$ 

It follows that 
$$\sum_{J\in\F_{n_0}(j)}\m(J)^2\ge\sum_{J\in\F_{n_0}(j)\ ;\ J\cap A\not=\emptyset}\frac1\e\ell^{-n_0\delta}\m(J)\ge\frac{1}{2\e}\ell^{-n_0\delta}\m(I_j).$$

Now, let $I\in\F_k$ and suppose that $I_{\e_1,\cdots,\e_k}$ with $\e_k=i$. Observe that if $J\in \F_{n}(j)$, then $\m(IJ)=\frac{p_{ij}}{\m(I_j)}\m(I)\m(J)$. If we choose $j_i\in\{0,\cdots,\ell-1\}$ such that $p_{ij_i}\not =0$, we get
\begin{eqnarray*}
\sum_{J\in\F_{n_0}}\m(IJ)^2&\ge&\sum_{J\in\F_{n_0}(j_i)}\m(IJ)^2\\
&\ge&\frac{p_{ij_i}^2}{\m(I_{j_i})^2}\m(I)^2\sum_{J\in\F_{n_0}(j_i)}\m(J)^2\\
&\ge&\frac{p_{ij_i}^2}{2\e\m(I_{j_i})}\ell^{-n_0\delta}\m(I)^2.
\end{eqnarray*}
Let $\e=\inf_i\left(\frac{p_{ij_i}^2}{4\m(I_{j_i})}\right)$ and the corresponding $n_0$. If $\eta$ is such that $\ell^{n_0\eta}=2$, we can rewrite the last inequality as
$$\sum_{J\in\F_{n_0}}\m(IJ)^2\ge\ell^{-n_0(\delta-\eta)}\m(I)^2.$$
If we sum this inequality on every interval $I$ of the same generation and iterate the process, we get for any $p\ge 1$
$$\sum_{I\in \F_{pn_0}}\m(I)^2\ge \ell^{-(p-1)n_0(\delta-\eta)}\sum_{I\in \F_{n_0}}\m(I)^2=C\ell^{-pn_0(\delta-\eta)}$$
which gives 
$$\tau(2)\ge -(\delta-\eta)>-\delta.$$
Moreover, it is clear that $\tau(\R)$ is not reduced to a single point. If follows that the measure $\m$ is multifractal.
\end{proof}

\bibliographystyle{amsplain}

\end{document}